\documentclass[10pt]{article}

\usepackage{amsmath,amsthm,amssymb,bbm}
\usepackage[bookmarksnumbered=true,pagebackref]{hyperref}
\usepackage{tocbibind}
\usepackage{enumitem}

\newtheorem{theorem}{Theorem}

\newtheorem{lemma}{Lemma}

\theoremstyle{remark}
\newtheorem{remark}{Remark}

\theoremstyle{definition}

\newtheoremstyle{notes}
{3pt}
{3pt}
{}
{}
{\bfseries}
{:}
{.4em}
{}
\theoremstyle{notes}
\newtheorem*{keywords}{Keywords}
\newtheorem*{subjclass}{AMS MSC 2020}

\newcommand{\StirlingI}[2]{\genfrac{[}{]}{0pt}{}{#1}{#2}}

\usepackage[disable]{todonotes} 
\title{Sylvester's problem for random walks and bridges}

\author{
Hugo Panzo
\\ 
\href{mailto:hugo.panzo@slu.edu}{\texttt{{\small hugo.panzo@slu.edu}}}
}

\date{\today}

\begin{document}

\maketitle

\begin{abstract}
Consider a random walk in $\mathbb{R}^d$ that starts at the origin and whose increment distribution assigns zero probability to any affine hyperplane. We solve Sylvester's problem for these random walks by showing that the probability that the first $d+2$ steps of the walk are in convex position is equal to $1-\frac{2}{(d+1)!}$. The analogous result also holds for random bridges of length $d+2$, so long as the joint increment distribution is exchangeable. 
\end{abstract}

\begin{keywords}
Convex hull, random bridge, random walk, Sylvester's problem. 
\end{keywords}

\begin{subjclass}
Primary 52A22, 60D05, 60G50; Secondary 52B05, 52B11.
\end{subjclass}


\section{Introduction}

The convex hull of any $d+1$ points in general position in $\mathbb{R}^d$ must be a $d$-simplex with each point being a vertex. Said differently, any $d+1$ points in general position in $\mathbb{R}^d$ are always in convex position. On the other hand, if we add an additional point while maintaining general position, two possibilities for the convex hull of these $d+2$ points now arise: a convex $d$-polytope with all $d+2$ points being vertices, or a $d$-simplex with only $d+1$ vertices. In other words, $d+2$ points in general position in $\mathbb{R}^d$ need not be in convex position.

When the $d+2$ points in general position in $\mathbb{R}^d$ are ``taken at random'', the problem of determining the probability that they are in convex position is often referred to as \emph{Sylvester's problem}. Sylvester's original formulation of the question, which concerned four points in the plane and first appeared in 1864 in \cite{Sylvester}, was ill-posed since the meaning of ``taken at random'' was never specified. This resulted in a variety of proposed answers, none of which could be determined to be \emph{the} answer; see \cite{Sylvester_history} for a historical account of the problem.

The most popular formulation of Sylvester's problem involves picking the $d+2$ points independently and uniformly at random from some convex set $K$ in $\mathbb{R}^d$. For example, Woolhouse and Sylvester, respectively, showed that the probability in question was $1-\frac{35}{12\pi^2}\approx 0.7045$ for disks and $\frac{2}{3}$ for triangles; see \cite{Sylvester_history}. As for higher dimensions, Kingman \cite[Theorem 7]{Kingman} was able to derive the expression 
\[
1-\left. \frac{d+2}{2^d}\binom{d+1}{\frac{1}{2}(d+1)}^{d+1} \middle/\binom{(d+1)^2}{\frac{1}{2}(d+1)^2} \right. 
\]
for the relevant probability when $K$ was a $d$-dimensional ball.

Another natural formulation of Sylvester's problem calls for picking the $d+2$ random points independently from the standard $d$-dimensional normal distribution. This problem was solved in the planar setting by Maehara \cite{Maehara}, and later on by Blatter \cite{Blatter}, who each used different methods to obtain the answer
\[
\frac{6}{\pi}\arcsin\left(\frac{1}{3}\right)\approx 0.6490.
\]
In the recent preprint \cite{Youden}, Frick, Newman, and Pegden were able to solve this Gaussian formulation of Sylvester's problem for $d=3$ by connecting it to \emph{Youden's demon problem} by way of \emph{Gale duality}. This allowed them to use an old result of David \cite{David} to deduce that the probability of five i.i.d.~standard trivariate Gaussian points in $\mathbb{R}^3$ being in convex position is
\[
\frac{1}{2}+\frac{5}{\pi}\arcsin\left(\frac{1}{4}\right)\approx 0.9022.
\]
Moreover, they used this connection to determine the large $d$ asymptotic of the complementary probability. That is, they showed that the probability of $d+2$ i.i.d.~standard multivariate Gaussian points in $\mathbb{R}^d$ \emph{not} being in convex position decays like
\[
2\sqrt{\frac{e^d}{2(d+2)^{d-2}(2\pi)^{d+1}}}~\text{ as }~d\to\infty.
\]

The present paper investigates an apparently new formulation of Sylvester's problem which is no less natural than the previous two. In this version, the $d+2$ random points are the initial steps of a random walk or bridge in $\mathbb{R}^d$. We solve this problem for all $d$ by deriving a simple formula for the probability that these points are in convex position. Remarkably, the formula is independent of the distribution of the increments of the walk---so long as it satisfies a mild condition which ensures that the steps are in general position. In fact, the same formula also holds for random bridges of length $d+2$, provided that the joint distribution of the increments is exchangeable. 

Even in the simple case of a random walk in $\mathbb{R}^2$ with i.i.d.~standard bivariate Gaussian increments, it seems difficult to derive our formula by direct calculation. Hence, in order to prove our theorems, we must appeal to some rather deep results of Kabluchko, Vysotsky, and Zaporozhets on convex hulls of random walks and bridges \cite{hyperplane,faces}; see Section \ref{sec:face_probabilities} for a description of the particular results that we use.


\section{Main results}

Our first main result concerns a random walk in $\mathbb{R}^d$ with independent and identically distributed increments. More precisely, let $X_1,X_2, \dots$ be an i.i.d.~sequence of random vectors in $\mathbb{R}^d$. Then the random walk $S=\{S_i:i\geq 0\}$ is defined by 
\[
S_i=X_1+\cdots +X_i,~~~i\geq 1,~~~S_0=0.
\]
The only assumption on the distribution of the increments is that it is sufficiently diffuse to avoid assigning positive probability to any affine hyperplane. 
\begin{enumerate}
\item[$(\text{Hy})$] \textit{Hyperplane condition:} $\mathbb{P}(X_1\in \mathcal{H})=0$ for every affine hyperplane $\mathcal{H}\subset \mathbb{R}^d$.
\end{enumerate}

\begin{theorem}\label{thm:walk}
Let $S$ be a random walk in $\mathbb{R}^d$ whose increment distribution satisfies the condition (Hy). Then the probability that the $d+2$ points $S_0, \dots, S_{d+1}$ are in convex position is $1-\frac{2}{(d+1)!}$.
\end{theorem}

\begin{remark}
As discussed in Remark \ref{rem:general_position}, the condition (Hy) is sufficient to ensure that the points $S_0,\dots,S_{d+1}$ are in general position almost surely.

\end{remark}

Our second main result is about random bridges in $\mathbb{R}^d$ of length $d+2$. These are basically random walks that start at the origin, take $d+1$ random steps, and then end back at the origin on step $d+2$. Let $\xi_1,\dots, \xi_{d+2}$ be a sequence of random vectors in $\mathbb{R}^d$, not necessarily independent. Then the random bridge $B=\{B_i:0\leq i\leq d+2\}$ of length $d+2$ is defined by 
\[
B_i=\xi_1+\cdots +\xi_i,~~~1\leq i\leq d+2,~~~B_0=0.
\]
Moreover, the following conditions are imposed on the joint distribution of the increments.
\begin{enumerate}
\item[$(\text{Br})$] \textit{Bridge property:} $B_{d+2}=\xi_1+\cdots +\xi_{d+2}=0$ almost surely;
\item[$(\text{Ex})$] \textit{Exchangeability:} For every permutation $\sigma$ of the set $\{1,\dots, d+2\}$, the distribution of $(\xi_{\sigma(1)},\dots, \xi_{\sigma(d+2)})$ is the same as that of $(\xi_1,\dots, \xi_{d+2})$;
\item[$(\text{GP})$] \textit{General position:} For every subsequence of indices $1\leq i_1<\cdots<i_d\leq d+1$ of length $d$, the probability that the vectors $B_{i_1},\dots,B_{i_d}$ are linearly dependent is $0$.
\end{enumerate}

\begin{theorem}\label{thm:bridge}
Let $B$ be a random bridge in $\mathbb{R}^d$ of length $d+2$, and whose joint increment distribution satisfies the conditions (Br), (Ex), and (GP). Then the probability that the $d+2$ points $B_1, \dots, B_{d+2}=B_0$ are in convex position is $1-\frac{2}{(d+1)!}$.
\end{theorem}

\todo{Discuss possible reasons for probabilities being the same.}

\begin{remark}
An interesting question is what further conditions, if any, are required in order to guarantee that conditioning (in the appropriate sense) a random walk which satisfies (Hy) to return to the origin on step $d+2$ results in a random bridge that satisfies (Ex) and (GP)?
\end{remark}


\section{Expected number of faces and face probabilities}\label{sec:face_probabilities}

As mentioned in the introduction, the proofs of Theorems \ref{thm:walk} and \ref{thm:bridge} rely on some impressive results of \cite{hyperplane,faces}. More specifically, for Theorem \ref{thm:walk}, we need a formula for the expected number of $0$-faces of the convex hull of the steps $S_0,\dots,S_n$ of a random walk in $\mathbb{R}^d$. Note that the number of $0$-faces of the convex hull is nothing but the number of vertices that it has. The full result, that is \cite[Theorem 1.2]{faces}, actually gives a formula for the expected number of $k$-faces for $0\leq k\leq d-1$. 

Let $\mathcal{C}_n$ denote the convex hull of the set of points $\{S_0,\dots,S_n\}$, and let $f_k(\mathcal{C}_n)$ denote the number of $k$-faces of $\mathcal{C}_n$. As long as the increment distribution of the random walk $S$ satisfies (Hy), then \cite[Remark 1.3]{faces} provides the desired formula, namely, 
\begin{equation}\label{eq:mean_vertices}
\mathbb{E}\left[f_0(\mathcal{C}_n)\right]=\frac{2}{n!}\left(\StirlingI{n+1}{d}+\StirlingI{n+1}{d-2}+\cdots\right).
\end{equation}
Here the $\StirlingI{n}{m}$'s denote \emph{signless Stirling numbers of the first kind}. Since we employ the convention that $\StirlingI{n}{m}=0$ whenever $m\leq 0$ or $m>n$, the sum appearing on the right-hand side of \eqref{eq:mean_vertices} has only finitely many nonzero terms. 

\begin{remark}\label{rem:general_position}
Some comments are in order on the assumptions that we make in the present paper versus the assumptions made in Theorem 1.2 of \cite{faces}. The random walks considered in \cite{faces} are allowed to have dependent increments. Consequently, a different set of assumptions are imposed. However, in case the increments are i.i.d., such as in the present paper, all that is needed is the condition (Hy); see \cite[Example 1.1]{faces}. Indeed, the proof of \cite[Proposition 2.5]{hyperplane} shows that for a random walk with i.i.d.~increments, the condition (Hy) is sufficient to ensure that any $d$ of the vectors $S_1,S_2,\dots$ will be linearly independent almost surely. 
\end{remark}

For Theorem \ref{thm:bridge}, we need a formula for the probability that a particular step of a random bridge in $\mathbb{R}^d$ with $d+2$ steps happens to be a $0$-face of the convex hull of that bridge. Let $\mathcal{F}_0(\mathcal{C}_{d+2})$ denote the set of $0$-faces of the convex hull of the points $B_1,\dots,B_{d+2}=B_0$. Then assuming that the joint increment distribution of the random bridge $B$ satisfies the conditions (Br), (Ex), and (GP), we have by Remark 1.13 of \cite{faces} the formula
\begin{equation}\label{eq:vertex_probability}
\mathbb{P}\big(B_i\in\mathcal{F}_0(\mathcal{C}_{d+2})\big)=\frac{2}{(d+2)!}\left(\StirlingI{d+2}{d}+\StirlingI{d+2}{d-2}+\dots\right),~0\leq i\leq d+2.
\end{equation}

Similarly to formula \eqref{eq:mean_vertices}, the sum appearing on the right-hand side of \eqref{eq:vertex_probability} has only finitely many nonzero terms. Moreover, it should be noted that the right-hand side is independent of the index $i$, that is, each step of $B$ has equal probability of being a vertex of the convex hull of $B$. This property can also be deduced from the cyclic exchangeability that results from imposing the condition (Ex); see \cite[Remark 1.13]{faces}. It should also be mentioned that the full result is much more general and actually computes the probability that the convex hull of any set of $k+1$ steps of $B$ happens to be a $k$-face of the convex hull of $B$; see \cite[Theorem 1.11]{faces}.


\section{Proofs}

The proofs of Theorems \ref{thm:walk} and \ref{thm:bridge} both require a simple identity for a certain sum of signless Stirling numbers of the first kind that is similar to the sums appearing in formulas \eqref{eq:mean_vertices} and \eqref{eq:vertex_probability}. Recall that for $0\leq m\leq n$, $\StirlingI{n}{m}$ counts the number of permutations of $\{1,\dots,n\}$ that contain exactly $m$ cycles; see \cite[Section 1.3]{StanleyI}. It follows that $\StirlingI{n}{n}+\StirlingI{n}{n-1}+\cdots+\StirlingI{n}{0}=n!$. The next lemma shows what happens to this sum if we exclude every other term. 

\begin{lemma}
Let $n\geq 2$ be an integer. Then we have
\begin{equation}\label{eq:StirlingI}
\StirlingI{n}{n}+\StirlingI{n}{n-2}+\cdots =\frac{n!}{2}.
\end{equation}
\end{lemma}

\begin{proof}
We prove \eqref{eq:StirlingI} using Proposition 1.3.7 of \cite{StanleyI}, which establishes the identity
\[
\sum_{k=0}^n\StirlingI{n}{k}t^k=\underbrace{t(t+1)\cdots (t+n-1)}_{\displaystyle R_n(t)}.
\]
Evidently, $R_n(1)=n!$ and $R_n(-1)=0$ when $n\geq 2$. Hence, for even $n\geq 2$, we have
\begin{align*}
\StirlingI{n}{n}+\StirlingI{n}{n-2}+\cdots &=\frac{1}{2}\left(\StirlingI{n}{n}+\StirlingI{n}{n-1}+\cdots +  \StirlingI{n}{n}-\StirlingI{n}{n-1}+\cdots\right)\\
&=\frac{1}{2}\big(R_n(1)+R_n(-1)\big)\\
&=\frac{n!}{2}.
\end{align*}
Similarly, when $n\geq 2$ is odd, we have
\begin{align*}
\StirlingI{n}{n}+\StirlingI{n}{n-2}+\cdots 
&=\frac{1}{2}\big(R_n(1)-R_n(-1)\big)\\
&=\frac{n!}{2}.
\end{align*}

\end{proof}

\subsection{Proof of Theorem \ref{thm:walk}}

\begin{proof}[Proof of Theorem \ref{thm:walk}]
Following Remark \ref{rem:general_position}, any $d$ of the vectors $S_1,\dots,S_{d+1}$ are linearly independent almost surely. Hence, any $d+1$ of the vectors $S_0,\dots,S_{d+1}$ are affinely independent. This shows that the points $S_0,\dots,S_{d+1}$ are in general position almost surely. As mentioned in the first paragraph of the Introduction, we now have two possibilities for the convex hull of these $d+2$ points: a convex $d$-polytope with all $d+2$ points being vertices, or a $d$-simplex with only $d+1$ vertices. Let $p$ denote the probability of the latter event. Then using the notation from Section \ref{sec:face_probabilities}, it is straightforward to deduce that
\begin{align}
\mathbb{E}\left[f_0(\mathcal{C}_{d+1})\right]&=(1-p)(d+2)+p(d+1)\nonumber \\
&=d+2-p.\label{eq:p_equation}
\end{align}

On the other hand, formula \eqref{eq:mean_vertices} and identity \eqref{eq:StirlingI} allow us to write 
\begin{align}
\mathbb{E}\left[f_0(\mathcal{C}_{d+1})\right]&=\frac{2}{(d+1)!}\left(\StirlingI{d+2}{d}+\StirlingI{d+2}{d-2}+\cdots\right)\nonumber \\
&=\frac{2}{(d+1)!}\left(\StirlingI{d+2}{d+2}+\StirlingI{d+2}{d}+\cdots\right)-\frac{2}{(d+1)!}\StirlingI{d+2}{d+2}\nonumber \\
&=d+2-\frac{2}{(d+1)!}.\label{eq:mean_equation}
\end{align}

Equating \eqref{eq:p_equation} and \eqref{eq:mean_equation} leads to
\[
p=\frac{2}{(d+1)!},
\]
from which the desired result follows.
\end{proof}

\subsection{Proof of Theorem \ref{thm:bridge}}

\begin{proof}[Proof of Theorem \ref{thm:bridge}]
By using the conditions (Br) and (GP) and repeating the argument from the proof of Theorem \ref{thm:walk}, we get that almost surely, the $d+2$ points $B_1,\dots, B_{d+2}=B_0$ are in general position. Moreover, with $p$ denoting the probability that the $d+2$ points $B_1,\dots, B_{d+2}=B_0$ are not in convex position, we similarly recover the formula
\begin{equation*}
\mathbb{E}\left[f_0(\mathcal{C}_{d+2})\right]=d+2-p.
\end{equation*}

Next, let $E_i$ denote the event $\{ B_i\in\mathcal{F}_0(\mathcal{C}_{d+2})\}$. Now we can use formula \eqref{eq:vertex_probability} and then repeat the calculations leading to \eqref{eq:mean_equation} using identity \eqref{eq:StirlingI} to deduce that
\begin{align*}
\mathbb{E}\left[f_0(\mathcal{C}_{d+2})\right]=\mathbb{E}\left[\sum_{i=1}^{d+2}\mathbbm{1}_{E_i}\right]&=\sum_{i=1}^{d+2}\mathbb{P}\big(B_i\in\mathcal{F}_0(\mathcal{C}_{d+2})\big)\\
&=(d+2)\Bigg(\frac{2}{(d+2)!}\left(\StirlingI{d+2}{d}+\StirlingI{d+2}{d-2}+\dots\right)\Bigg)\\
&=\frac{2}{(d+1)!}\left(\StirlingI{d+2}{d}+\StirlingI{d+2}{d-2}+\dots\right)\\
&=d+2-\frac{2}{(d+1)!}.
\end{align*}

Just like in the proof of Theorem \ref{thm:walk}, it follows that the probability the $d+2$ points $B_1,\dots, B_{d+2}=B_0$ are in convex position is equal to $1-\frac{2}{(d+1)!}$.
\end{proof}

\bibliography{Sylvester_bib}
\bibliographystyle{amsplainabbrev}

\end{document}